\documentclass[10pt]{amsart}

\usepackage{amsmath,graphics}
\usepackage{amsfonts,amssymb}
\usepackage{amscd}
\usepackage{mathrsfs}
\usepackage[all]{xy}
\usepackage{accents}
\usepackage{array}
\usepackage{tikz}
\usepackage{fancyvrb}
\usepackage{enumitem}

\theoremstyle{plain}
\newtheorem*{theorem*}{Theorem}
\newtheorem*{lemma*} {Lemma}
\newtheorem*{corollary*} {Corollary}
\newtheorem*{proposition*} {Proposition}
\newtheorem{theorem}{Theorem}[section]
\newtheorem{lemma}[theorem]{Lemma}
\newtheorem{corollary}[theorem]{Corollary}
\newtheorem{proposition}[theorem]{Proposition}

\newtheorem{prelemma*}{Prelemma}
\newtheorem{prelemma}[theorem]{Prelemma}
\theoremstyle{remark}
\newtheorem*{remark}{Remark}
\newtheorem*{definition}{Definition}
\newtheorem*{notation}{Notation}
\newtheorem*{example}{Example}

\theoremstyle{definition}

\def \R {\mathbb{R}}

\def \Z {\mathbb{Z}}
\def \C {\mathbb{C}}
\def \P {\mathbb{P}}

\def \Qt{\widetilde{Q}}
\def \Qtl{\widetilde{Q}_\lambda}
\def \Pt{\widetilde{P}}
\def \Ptl{\widetilde{P}_{\lambda}}

\def \Dn{\mathcal{D}(n)}

\def \bn{\begin{enumerate}}
\def \en{\end{enumerate}}
\def \bdm{\begin{displaymath}}
\def \edm{\end{displaymath}}

\def \bp{\begin{proof}}
\def\ep{\end{proof}}

\def\vd{\vdots}
\def\dd{\ddots}
\def\U+e{(U^+)^e}

\def\trm{\hspace{0.08in}\textrm{for} \hspace{0.08in}}
\def\be{\begin{equation}}
\def\ee{\end{equation}}

\def\Fdot{ F_{\mbox{\boldmath{.}}}}

\def\hs{\hspace{0.07in}}

\def\og{OG(n)}
\def\lg{LG(n)}
\def\In{\mathcal{I}_n}

\def\dn{\mathcal{D}(n)}

\def\mo{\mathcal{O}}

\begin{document}

\title [QUANTUM MULTIPLICATION OPERATORS]{QUANTUM MULTIPLICATION OPERATORS FOR LAGRANGIAN AND ORTHOGONAL GRASSMANNIANS}
\author{Daewoong Cheong}
\address{Department of Mathematical Sciences, KAIST\\
291 Daehak-ro, Yuseong-gu, Daejeon 34141, Korea}
\email{daewoongc@kias.re.kr}
\date{Nov. 10, 2014}

\begin{abstract}
In this article, we make a close analysis on quantum multiplication
operators on the quantum cohomology rings  of Lagrangian and
orthogonal Grassmannians, and give an explicit description on all
simultaneous eigenvectors and the corresponding eigenvalues for
these operators. As a result, we show that Conjecture $\mathcal{O}$
of Galkin, Golyshev and Iritani holds for these manifolds.
\end{abstract}
\maketitle

\section{INTRODUCTION}
Let $M$ be a  Fano manifold. The quantum cohomology ring
$qH^*(M,\C)$ is a certain deformation of the classical cohomology
ring $H^*(M,\C)$.  For $\sigma \in qH^*(M,\C)$, define the quantum
multiplication operator $[\sigma]$ on $qH^*(M)$ by
$[\sigma](\alpha)=\sigma\cdot \alpha.$ Denote the set of eigenvalues
of $[\sigma]$ by $\mathrm{Spec} ([\sigma])$. Suppose
$\mathrm{Spec}([c_1(M)])=\{a_1,...,a_m\},$ and
$T_0=\mathrm{Max}\{|a_1|, |a_1|,...,|a_m|\}$,  where
$c_1(M)=c_1(TM)$ is the first Chern class of the tangent bundle $TM$
of $M.$ Then we say that $M$ satisfies  Conjecture  $\mathcal{O}$ if
\bn
\item $T_0$ is an eigenvalue of $[c_1(M)]$.
\item  If $u$ is an eigenvalue of  $[c_1(M)]$ such that $|u|=T_0,$ then $u=T_0\xi $ for some $r$-th root of unity, where $r$ is the Fano index of $M.$
    \item The multiplicity of the eigenvalue $T_0$ is one.
\en

Originally, while Galkin, Golyshev and Iritani studied the
exponential asymptotic of solutions to the quantum differential
equations, they proposed two more conjectures called Gamma
Conjectures  $\mathrm{I}$, $\mathrm{II}$ which, informally, relate
the quantum cohomology of $M$ and the so-called Gamma class in terms
of differential equations. We refer to \cite{GGI} for details on
these materials. The importance of Conjecture $\mathcal{O}$ lies in
the fact that it `underlies' the Gamma Conjectures. Indeed, above
all,  the Gamma Conjecture $\mathrm{I}$ was stated under the
Conjecture $\mo$.  And under further assumption of the
semisimplicity of the quantum cohomology of $M$, the Gamma
Conjecture $\mathrm{II}$, which is a refinement of a part of
Dubrovin's conjecture (\cite{Du}), relates eigenvalues of the
operator $[c_1(M)]$ with members of a certain exceptional collection
 of the derived category
$\mathcal{D}^b_{\mathrm{coh}}(M)$ bijectively. Then, under the
semisimplicity of the quantum cohomology ring, the Gamma Conjecture
$\mathrm{I}$ can be viewed as a part of the Gamma conjecture II; it
relates the eigenvalue $T_0$ to the member $\mathcal{O}_M$ of the
aforementioned exceptional collection.

Let us mention for which manifolds Conjecture $\mo$ has been proved.
The Grassmannian  is the first manifold for which  Conjecture
$\mathcal{O}$ has been proved.  Indeed, in \cite{GGI}, Galkin,
Golyshev and Iritani recently proved Conjecture $\mathcal{O}$, and
then the Gamma Conjectures $\mathrm{I}$, $\mathrm{II}$ for the
Grassmannian.  In fact,  we noticed that for the Grassmannian,
Galkin and Golyshev  gave a very short proof of Conjecture
$\mathcal{O}$ in an earlier paper of theirs (\cite{GG1}), and
Rietsch gave an explicit description on the eigenvalues and
eigenvectors of multiplication operators on the quantum cohomology
ring of the Grassmannian (\cite{Riet1}) which contains all the
ingredient necessary to prove Conjecture $\mo$ for the Grassmannian.
For toric Fano manifolds, assuming Conjecture $\mathcal{O},$ Galkin,
Golyshev and Galkin proved Gamma conjectures $\mathrm{I}$ and
$\mathrm{II}$.

In this article, following Rietch, we obtain simultaneous
eigenvectors and corresponding eigenvalues for multiplication
operators on the quantum cohomology rings (specialized at $q=1$) of
Lagrangian and orthogonal Grassmannians,  which are homogeneous
varieties and in particular examples of Fano manifolds. Then we use
these to show that Conjecture $\mathcal{O}$ holds for Lagrangian and
orthogonal Grassmannians. Very recently, Li and the author worked
out Conjecture $\mo$ for general homogeneous varieties by a
different approach (\cite{CL1}).

Lastly, let us  explain what makes possible for these manifolds such
an explicit description of the  multiplication operators. In this
article, we heavily use one of  Peterson's results, which states
that the quantum cohomology ring of a homogeneous variety is
isomorphic with the coordinate ring of so-called Peterson variety
corresponding to the homogeneous variety (\cite{Pete1},
\cite{Pete2}, \cite{Riet2}, \cite{CH1}). Unlike general homogeneous
varieties, for these manifolds together with the Grassmannian, there
is a much simpler isomorphic variety that can replace the Peterson
variety. Thereby we identify points of the Peterson variety
(specialized at $q=1$). On the other hand, as the Grassmannian does,
they have symmetric polynomials representing the Schubert classes
and the quantization of these polynomials which serve as  regular
functions on the Peterson variety, too. These two facts provide us
with orthogonal formulas evaluated at  points of the Peterson
variety (cut out by $q=1$) which play a key role in explicitly
finding simultaneous eigenvectors and the associated eigenvalues.
Much of material needed in this article was studied in the author's
earlier paper \cite{CH1}. {\it \textbf{Acknowledgements}.} This
research was supported by Basic Science Research Program through the
National Research Foundation of Korea(NRF) funded by the Ministry of
Education(NRF-2016R1A6A3A11930321), and also  by the Ministry of
Science, ICT and Future Planning(NRF-2015R1A2A2A01004545).

\section{SYMMETRIC FUNCTIONS}
In this section, we review $\Qt$- and $\Pt$-polynomials of Pragacz
and Ratajski and  Schur polynomials.  References are \cite{Pra and
Rat 1} and \cite{Las and Pra} for the former polynomials, and
\cite{F} and \cite{Mac1} for the latter.
\subsection{Notations and definitions}
We begin with some notations concerning the labeling of symmetric
polynomials. A $partition$ $\lambda$ is a weakly decreasing sequence
$\lambda=(\lambda_1,\lambda_2,...,\lambda_m)$ of nonnegative
integers. A $Young \hs diagram $ is a collection of boxes, arranged
in left-justified rows, with a weakly decreasing number of boxes in
each row. To a partition $\lambda=(\lambda_1,...,\lambda_m),$ we
associate a Young diagram whose $i$-th row has $\lambda_i$ boxes.
The nonzero $\lambda_i$ in $\lambda=(\lambda_1,...,\lambda_m)$ are
called the $parts$ of $\lambda.$ The number of the parts of
$\lambda$ is called the $length$ of $\lambda,$ denoted by
$l(\lambda)$; the sum of the parts of $\lambda$ is called the
$weight$ of $\lambda$, denoted by $|\lambda|.$ For
$\lambda=(\lambda_1,...,\lambda_m)$  with $l(\lambda)=l$,  we
usually write $(\lambda_1,...,\lambda_l)$ for
$(\lambda_1,...,\lambda_l,0,...,0)$ if any confusion does not arise.
For positive integers $m$ and  $n,$ denote by $\mathcal{R}(m,n)$ the
set of all partitions whose Young diagram fits inside an $m\times n$
diagram, which is the Young diagram of the partition $(n^m).$ A
partition $\lambda=(\lambda_1,...,\lambda_l,...,\lambda_m)\in
\mathcal{R}(m,n)$ is called strict if $\lambda_1>\cdots>\lambda_l$
and $\lambda_{l+1}=\cdots \lambda_m=0$, where $l=l(\lambda).$ Denote
by $\mathcal{D}(m,n)$ the set of all strict partitions in
$\mathcal{R}(m,n).$ If $m=n$, then we write $\mathcal{R}(n)$ and
$\mathcal{D}(n)$ for $\mathcal{R}(n,n)$ and $\mathcal{D}(n,n)$,
respectively. If $\lambda \in \mathcal {D}(n)$, denote by
$\widehat{\lambda}$ the partition whose parts complements the parts
of $\lambda$ in the set $\{1,...,n\}$.  For $\lambda \in
\mathcal{R}(m,n),$ the $conjugate$ of $\lambda$ is the partition
$\lambda^{t} \in \mathcal{R}(n,m)$ whose Young diagram is the
transpose of that of $\lambda.$

\subsection{Symmetric functions}
Let $X:=(x_1,...,x_n)$ be the $n$-tuple of variables. For
$i=1,...,n,$ let $H_i(X)$ (resp. $E_i(X)$) be the $i$-th complete
(resp. elementary) symmetric function. Then for any partition
$\lambda,$ the Schur polynomial $S_{\lambda}(X)$ is defined by
$$S_{\lambda}(X):=\textrm{Det}[H_{\lambda_i+j-i}(X)]_{1\leq i,j
\leq n }=\textrm{Det}[E_{\lambda_{i}^t+j-i}(X)]_{1\leq i,j \leq n
},$$ where $H_0(X)=E_0(X)=1$ and $H_{k}(X)=E_k(X)=0$ for $k<0.$

We define $\Qt$- and $\Pt$-polynomials both of which are indexed by
elements of $\mathcal{R}(n)$. For $i=1,...,n,$  set
$\Qt_i(X):=E_i(X)$, the $i$-th elementary symmetric function.

Given two nonnegative integers $i$ and $ j$ with $i\geq j$, define
\begin{displaymath}
\Qt_{i,j}(X)=\Qt_i(X)\Qt_j(X)+2\sum_{k=1}^{j}(-1)^k
\Qt_{i+k}(X)\Qt_{j-k}(X).
\end{displaymath}

 Finally, for any partition
$\lambda$ of length $l=l(\lambda)$, not necessarily strict, and for
$r=2\lfloor(l+1)/2\rfloor,$ let $B_{\lambda}^r=[B_{i,j}]_{1\leq
i,j\leq r}$  be the skew symmetric matrix defined by
$B_{i,j}=\Qt_{\lambda_i,\lambda_j}(X)$ for $i<j$. We set
\begin{displaymath}
\Qt_{\lambda}(X)=\textrm{Pfaffian}(B_{\lambda}^r).
\end{displaymath}
Given $\lambda$, not necessarily strict,  $\Ptl$ is defined by
$$\Ptl(X):=2^{-l(\lambda)}\Qtl(X).$$

\begin{proposition}(\cite{Pra and Rat 1})\label{property-of-Qt-poly}
The $\Qt$-polynomials satisfy the following properties.
\begin{enumerate}

\item For $i=1,...,n,$  $\Qt_{i,i}(X)=E_i(x_1^2,...,x_n^2).$
\item
For any $\lambda \in \mathcal{D}(n),$
$$\Qtl(X)\Qt_n(X)=\Qt_{(n,\lambda_1,...,\lambda_l)}(X).$$
\label{fac:Qspe}
\end{enumerate}
\end{proposition}

Let $S_n=<s_1,...,s_{n-1}>$ be the symmetric group generated by the
simple transpositions $s_1,...,s_{n-1}$. The Weyl group $W_n$ for
Lie type $C_n$ is an extension of the symmetric group $S_n$ by $s_0$
such that the following relations hold
$$s_0^2=1, \hspace{0.1in}s_0\cdot s_1 \cdot s_0 \cdot s_1=s_1 \cdot s_0\cdot s_1\cdot s_0,
\hspace{0.1in} s_0 \cdot s_i=s_i\cdot s_0 \hspace{0.09in}
\mathrm{for}\hspace{0.09in} i\geq 2.$$ Recall that an element $w$ of
$S_n$ can be represented by a sequence $w=(w_1,....w_n)$ of numbers
$1,...,n$,  e.g., $s_i=(1,...,i,i+1,i, i+2,...,n)$  for
$i=1,...,n-1.$ In contrast, an element of  $W_n$ can be represented
by a permutation with bars $w=(w_1,...,w_n)$, e.g.,
$s_0=(\bar{1},2,...,n)$. With this notation, the multiplication in
$W_n$ is given as follows: For $(w_1,...,w_n)\in W_n$,
$$(w_1,...,w_n)\cdot s_0=(\bar{w}_1,w_2,...,w_n);$$
$$(w_1,...,w_n)\cdot
s_i=(w_1,...,w_{i-1},w_{i+1},w_i,w_{i+2},...,w_n) \hspace{0.1in}
\mathrm{for}\hspace{0,07in} i=1,...,n-1.$$

For $w=(w_1,...,w_n)\in W_n$ and   $X=(x_1,...,x_n)$, let $X^w$ be
the $n$-tuple $(y_1,...,y_n)$ of `signed variables' $\pm x_1,...,\pm
x_n$, where $y_k=x_{w_k}$ ( resp. $-x_{w_k}$) if $w_k$ is unbarred
(resp. barred). This induces an action of $W_n$ on the ring of
polynomials in $x_1,...,x_n$, which is defined as follows: For $w\in
W_n$ and a polynomial $P(X)$ in $x_1,...,x_n$,  $$w\cdot
P(X):=P(X^w).$$

The following will be used to derive orthogonal formulas for the
Lagrangian and orthogonal Grassmannians later.
\begin{proposition}(\cite{Las and Pra})\label{Key}
 We have the following identity of polynomials
\begin{displaymath}\sum_{\lambda\in \mathcal{D}(n)}\Pt_{\lambda}(X^w)\Pt_{\hat{ \lambda}}(X)=\left\{\begin{array}{cc}
S_{\rho_n}(X)& \textrm{if} \hspace{0.05in} w \in S_n,\\
0  & \textrm{if} \hspace{0.05in} w \in W_n\setminus S_n.
\end{array}\right.
\end{displaymath}
\end{proposition}

Let  $\Lambda_n$  be the algebra over $\Z$ of symmetric functions in
$x_1,...,x_n$. Let $\Lambda_n^\prime$ the algebra over $\Z$ of
symmetric polynomials  generated by $\Pt_{\lambda}$ with $\lambda
\in \mathcal{R}(n).$ Note that $\Lambda_n$ is spanned by the
polynomials $\Qt_\lambda$ with $\lambda$ with $\lambda \in
\mathcal{R}(n)$, and so $\Lambda_n^\prime$ is isomorphic to
$\Lambda_n$ as $\Z$-modules.

\section{QUANTUM COHOMOLOGY RINGS}

\subsection{Lagrangian and orthogonal Grassmannians}

Let  $E=\C^N$ be a complex vector space equipped with a
nondegenerate (skew) symmetric bilinear form $Q.$ A subspace
$W\subset E$ is called $isotropic$ if $Q(v,w)=0$ for all $v,w \in
W.$ A maximal isotropic subspace of $E$ is of (complex) dimension
$\lfloor\frac{N}{2}\rfloor$. In particular, when $N=2n$ and $Q$ is a
skew symmetric form, such a maximal subspace is called {\it
Lagrangian}. Let $\lg$ be the parameter space of Lagrangian
subspaces in $E$. Then $\lg$ is a homogeneous variety
$Sp_{2n}(\C)/P_n$ of complex dimension $n(n+1)/2$, where $P_n$ is
the maximal parabolic subgroup of the symplectic group $Sp_{2n}(\C)$
associated with the `right end root' in the Dynkin Diagram of Lie
type $C_n$,  e.g., on Page $58$ of \cite{Hu1}.

For the case when $N=2n+1$ and $Q$ is a symmetric bilinear form, let
$OG(n)$ be the parameter space of maximal isotropic subspaces in
$E.$  Then $OG(n)$ is a homogeneous variety $SO_{2n+1}(\C)/P_{n}$ of
dimension $n(n+1)/2$, where $P_{n}$ is the maximal parabolic
subgroup of $SO_{2n+1}(\C)$ associated with a `right end root' of
the Dynkin diagram of type $B_n$   (on Page $58$ of \cite{Hu1}).
Traditionally, the manifold $OG(n)$ is called an $odd$ $orthogonal$
$Grassmannian$. There is an `even counterpart', $even$ $orthogonal$
$Grassmannian,$ written as $SO_{2n+2}(\C)/P_{n+1}$, where $P_{n+1}$
is the maximal parabolic subgroup of $SO_{2n+2}(\C)$ associated with
the `right end root' of the Dynkin diagram of Lie type $D_{n+1}$ (on
Page $58$ of \cite{Hu1}). It is well-known that they are isomorphic
(projectively equivalent) to each other. Therefore, in this paper,
we treat only $OG(n)$, and we will go without the adjective `odd'.

\subsection{Quantum cohomology of $LG(n)$} To describe the quantum cohomology of $\lg,$ we begin with the definition of
Schubert varieties of $\lg$. Given a complex vector space $E$ of
dimension $2n$ with a nondegenerate skew-symmetric form, fix a
complete isotropic flag $ F_{\mbox{\boldmath{.}}}$ of subspaces
$F_i$ of $E$:
$$\Fdot : 0=F_0\subset F_1\subset \cdots \subset F_n\subset E,$$
where dim$(F_i)=i$  for each $i,$ and $F_n$ is Lagrangian. To
$\lambda \in \Dn,$ we associate the Schubert variety
$X_{\lambda}(\Fdot)$  defined as the locus of $ \Sigma \in \lg$ such
that \be \label{schubert} \textrm{dim}(\Sigma \cap
F_{n+1-\lambda_i})\geq i \trm i=1,...,l(\lambda).\ee Then
$X_{\lambda}(\Fdot)$ is a subvariety of $\lg$ of complex codimension
$| \lambda |.$ The Schubert class associated with $\lambda$ is
defined to be the cohomology class, denoted by $\sigma_{\lambda}$,
Poincar\'{e} dual to the homology class $[X_{\lambda}(\Fdot)]$, so
 $\sigma_\lambda \in H^{2|\lambda|}(\lg,\Z).$ It is a classical
result that
$\{\sigma_{\lambda}\hspace{0.05in}|\hspace{0.05in}\lambda\in \dn\}$
forms an additive basis for $H^{*}(\lg,\Z)$. It is conventional to
write $\sigma_i$ for $\sigma_{(i)}$. If $$0\rightarrow S\rightarrow
\mathcal{E}\rightarrow Q\rightarrow 0$$ denotes the short exact
sequence of tautological vector bundles on $LG(n)$, then $\sigma_i$
equals the $i$-th Chern class $c_i(Q)$ of the tautological quotient
bundle $Q$ on $LG(n)$, where $\mathcal{E}$ denotes the trivial
bundle $\mathcal{E}=LG(n)\times \C^{2n}.$ It is known that there is
a surjective ring homomorphism from $\Lambda_n \rightarrow
H^*(\lg,\Z)$ sending $\Qt_\lambda(X)$ to $\sigma_\lambda$ which has
the kernel generated by $\Qt_{i,i}$ with $i=1,...,n.$

A rational map of degree $d$ to $\lg$ is a morphism
$f:\mathbb{P}^1\rightarrow \lg$ such that
$$\int_{LG(n)}f_*[\P^1]\cdot \sigma_1=d.$$ Given an integer $d\geq0$
and partitions $\lambda, \mu,$ $\nu\in \mathcal{D}(n)$, the
Gromov-Witten invariant
$<\sigma_{\lambda},\sigma_{\mu},\sigma_{\nu}>_d$  is defined as the
number of rational maps $f:\mathbb{P}^1\rightarrow \lg$ of degree
$d$ such that $f(0)\in X_{\lambda}(\Fdot),$ $f(1)\in X_{\mu}(
G_{\mbox{\boldmath{.}}}),$ and $f(\infty)\in X_{\nu}(
H_{\mbox{\boldmath{.}}})$, for given isotropic flags $\Fdot$,
$G_{\mbox{\boldmath{.}}},$ and $H_{\mbox{\boldmath{.}}}$ in general
position. We remark that
$<\sigma_{\lambda},\sigma_{\mu},\sigma_{\nu}>_d$ is $0$ unless
$|\lambda|+|\mu|+|\nu|=\mathrm{dim}(LG(n))+(n+1)d$. The quantum
cohomology ring $qH^{*}(\lg,\Z)$ is isomorphic
 to $H^{*}(\lg,\Z)\otimes \Z[q]$
as  $\Z[q]$-modules, where $q$ is a formal variable of degree
$(n+1)$ and called the $quantum$ $variable$. The multiplication in
$qH^{*}(\lg,Z)$ is given by the relation \be~\label{multi}
\sigma_\lambda \cdotp
\sigma_\mu=\sum<\sigma_\lambda,\sigma_\mu,\sigma_{\hat{\nu}}>_d
\sigma_\nu q^d,\ee  where the sum is taken over $d\geq0$ and
partitions $\nu$ with $|\nu|=|\lambda|+|\mu|-(n+1)d.$\\

Set $X^+:=(x_1,...,x_{n+1})$, and let $\tilde{\Lambda}_{n+1}$ be the
subring of $\Lambda_{n+1}$ generated by the polynomials $\Qt_i(X^+)$
for $i\leq n$ together with the polynomial $2\Qt_{n+1}(X^+).$

Now we are ready to give a presentation of the quantum cohomology
ring of $\lg$ and the quantum Giambelli formula due to Kresch and
Tamvakis.
\begin {theorem}[\cite{KT2}] \label{quantum coho:lg}
There is a surjective ring homomorphism from $\tilde{\Lambda}_{n+1}$
to $qH^*(\lg,\Z)$ sending $\Qt_\lambda(X^+)$ to $\sigma_\lambda$ for
all $\lambda \in \mathcal{D}(n)$ and $2\Qt_{n+1}(X^+)$ to $q$, with
the kernel generated by $\Qt_{i,i}$ for $1\leq i \leq n.$ The ring
$qH^*(\lg,\Z)$ is presented as a quotient of the polynomial ring
$\Z[\sigma_{1},...,\sigma_n, q]$ by the relations
$$\sigma_i^2+2\sum_{k=1}^{n-i}(-1)^k\sigma_{i+k}\sigma_{i-k}=(-1)^{n-i}\sigma_{2i-n-1}q$$
for $1\leq i \leq n.$ The Schubert class in this presentation is
given by quantum Giambelli formula
\begin{displaymath}\sigma_{i,j}=\sigma_{i}\sigma_j+2\sum_{k=1}^{n-i}(-1)^k\sigma_{i+k}
\sigma_{j-k}+(-1)^{n+1-i}\sigma_{i+j-n-1}q\end{displaymath} for
$i>j>0,$ and

\begin{displaymath}\sigma_\lambda=\mathrm{Pfaffian}(\sigma_{\lambda_i,\lambda_j})_{1\leq <i<j\leq r},
\end{displaymath}
where quantum multiplication is employed throughout.

\end{theorem}

See \cite {KT2} for more details on the quantum cohomology ring of
$\lg.$

\subsection{Quantum cohomology  of Orthogonal  Grassmannian}\label{subsec:ogo}
The quantum cohomology theory of $OG(n)$ is parallel with that of
$\lg$. Let $E$ be a complex vector space of dimension $2n+1$
equipped with a nondegenerate symmetric form. Given $\lambda \in
\dn$, the Schubert variety $X_{\lambda}(\Fdot)$  is defined by the
same equation $(\ref{schubert})$ as before, relative to an isotropic
flag $ F_{\mbox{\boldmath{.}}}$ in $E.$  The Schubert class
$\tau_\lambda$ is defined as a cohomology class Poincar\'{e} dual to
$[X_{\lambda}(\Fdot)]$. Then $\tau_\lambda \in
H^{2l(\lambda)}(OG(n),\Z)$, and the cohomology classes
$\tau_\lambda$, $\lambda \in \mathcal{D}(n)$, form a $\Z$-basis for
$H^*(\og,\Z).$ The cohomology ring $H^*(\og,\Z)$ can be presented in
terms of $\Pt$- polynomials. More precisely, there is a surjective
ring homomorphism from $\Lambda_n^\prime$ to $H^{*}(\og,\Z)$ sending
$\Ptl(X)$ to $\tau_{\lambda}$ is a surjective ring homomorphism with
the kernel generated by the polynomials $\Pt_{i,i}(X)$ for all
$i=1,...,n$.
\smallskip

 For $OG(n)$, the Gromov-Witten invariants are defined similarly.
 Given an integer $d\geq 0$, and $\lambda,\mu, \nu \in \mathcal{D}(n)$, the Gromov-witten invariant
$<\tau_\lambda, \tau_\mu, \tau_\nu>_d$ is defined as  the number of
rational maps $f:\mathbb{P}^1\rightarrow \og $ of degree $d$ such
that $f(0)\in X_{\lambda}(\Fdot),$ $f(1)\in X_{\mu}(
G_{\mbox{\boldmath{.}}}),$ and $f(\infty)\in X_{\nu}(
H_{\mbox{\boldmath{.}}})$, for given isotropic flags $\Fdot$,
$G_{\mbox{\boldmath{.}}},$ and $H_{\mbox{\boldmath{.}}}$ in general
position. Note that $<\tau_\lambda, \tau_\mu, \tau_\nu>_d=0$ unless
$|\lambda|+|\mu|+|\nu|=\textrm{deg}(\og)+2nd$. The quantum
cohomology ring of $\og$ is isomorphic to $H^*(\og,\Z)\otimes \Z[q]$
as  $\Z[q]$-modules. The multiplication in $qH^*(\og,\Z)$ is given
by the relation \be~\label{multi:2} \tau_\lambda \cdotp
\tau_\mu=\sum<\tau_\lambda,\tau_\mu,\tau_{\hat{\nu}}>_d\tau_\nu
q^d,\ee where the sum is taken over $d \geq0$ and partitions $\nu$
with $|\nu|=|\lambda|+|\mu|-2nd.$ We note that the degree of the
quantum variable $q$ in $qH^*(OG(n),\Z)$ is $2n,$ whereas that of
$q$ in $qH^*(\lg,Z)$ is $(n+1).$

 \begin{theorem}[\cite{KT1}]\label{quantum coho:oge}
There is a surjective ring homomorphism from $\Lambda_n^\prime$ to
$qH^*(\og,\Z)$ sending $\Pt_\lambda(\textsc{X})$ to $\tau_\lambda$
for all $\lambda \in \mathcal{D}(n)$ and $\Pt_{n,n}(X)$ to $q$, with
the kernel generated by $\Pt_{i,i}$ for $1\leq i \leq n-1.$ The
quantum cohomology ring $qH^{*}(\og,\Z)$ is presented as a quotient
of the polynomial ring $\Z[\tau_1,...,\tau_n,q]$ modulo the
relations $\tau_{i,i}=0$ for $i=1,...,n-1$ together with the quantum
relation $\tau_n^2=q$, where \be \label{tau_ii}
\tau_{i,i}:=\tau_i^2+2\sum_{k=1}^{i-1} (-1)^k \tau_{i+k} \tau_{i-k}
+ (-1)^i \tau_{2i}.\ee The Schubert class $\tau_\lambda$ in this
presentation is given by the quantum Giambelli formulas
\begin{displaymath}
\tau_{i,j}=\tau_i\tau_j+2\sum_{k=1}^{j-1}(-1)^k\tau_{i+k}\tau_{j-k}+(-1)^j\tau_{i+j}
\end{displaymath} for $i>j>0,$ and
\begin{displaymath}
\tau_\lambda=\mathrm{Pfaffian}(\tau_{\lambda_i,\lambda_j})_{1\leq
i<j\leq r},
\end{displaymath}
where quantum multiplication is employed throughout.

\end{theorem}

See \cite{KT1} for more details on the quantum cohomology of $\og$.
\subsection{Quantum Euler class}\label{subsection-quan-euler}
The quantum Euler class $e_q(M)$ of a projective manifold $M$ is a
deformation of (ordinary) Euler class  $e(M)$. Originally, it is
defined in the context of the so-called Frobenius algebra(\cite{A}).
Restricting ourselves to the cases $OG(n)$ and $LG(n)$ for
simplicity, the quantum Euler class $e_q(M)$ can be defined as
follows.
\begin{definition}\label{def-quantum euler} For $M=OG(n)$ or
$LG(n)$, the quantum Euler classes $e_q(M)$ are respectively defined
as
 \be
 e_q(OG(n)):=\sum_{\lambda\in
\mathcal{D}(n)}\tau_\lambda\cdot\tau_{\hat{\lambda}},\ee

\be \hspace{0.1in}\hspace{0.1in} e_q(LG(n)):=\sum_{\lambda\in
\mathcal{D}(n)}\sigma_\lambda\cdot\sigma_{\hat{\lambda}}.\ee
\end{definition}

Note that if we replace the quantum product in the above definitions
by the ordinary product in $H^*(M)$, we get  the  Euler class $e(M)$
of $M.$ The object $e_q(M)$ encodes  information on the
semisimplicity of the quantum cohomology ring as follows.

\begin{proposition}\label{euler cri}(\cite{A}, Theorem 3.4) For a projective manifold $M$,
the quantum cohomology ring $qH^*(M)$ with the quantum parameters
specialized to nonzero complex numbers is semisimple if and only if
the quantum Euler class (after the specialization) is invertible in
that ring.
\end{proposition}
There is a class of  manifolds whose quantum cohomology rings are
semisimple.

\begin{proposition}\label{semisimplicity}(\cite{CMP}) If $M=G/P$ is a minuscule or
cominuscule homogeneous variety, then the quantum cohomology ring
$qH^*(M)$, which contains a single quantum variable $q$, is
semisimple after specializing at $q=1$.
\end{proposition}

 We refer to \S $2$ of \cite{CP1} for the notion of minuscule or cominuscule homogeneous varieties of Proposition \ref{semisimplicity}.
Note that $OG(n)$ is minuscule and $LG(n)$ is cominuscule (see \S
$2$ of \cite{CP1}). Thus, the rings $qH^*(OG(n))_{q=1}$ and
$qH^*(LG(n))_{q=1}$ are semisimple.

\section{Peterson's result}
\subsection{Peterson's result}
Informally, one of  Peterson's (unpublished) results on the quantum
cohomology
 can be stated as follows (\cite{Pete1}): Let $G$ be a semisimple algebraic group and $B $  a Borel subgroup of $G$.
 Let $G^\vee$ be the Langlands dual of $G,$ and $B^\vee$  a Borel subgroup of $G^\vee.$ For a parabolic subgroup $P$ of $G$ containing $B$, the quantum cohomology ring of the homogeneous variety
$G/P$ is isomorphic with the coordinate ring
$\mathcal{O}(\mathcal{Y}_P)$ of a (an affine) subvariety
$\mathcal{Y}_P$, which is a stratum of so-called Peterson's variety
$\mathcal{Y}\subset G^\vee/B^\vee$, i.e.,
$$\mathcal{Y}=\bigcup_{Q}\mathcal{Y}_Q,$$ where $Q$ ranges over parabolic subgroups containing $B.$
For convenience, we will simply call the subvariety $\mathcal{Y}_P$
a Peterson variety (corresponding to $P)$, too, if there is no
confusion. When $P$ is a minuscule parabolic subgroup of $G,$ this
Peterson's result goes further (\cite{Pete2}). More precisely, in
this case, the variety $\mathcal{Y}_P$ can be replaced by a simpler
isomorphic variety $\mathcal{V}_P\subset U^\vee$, where $U^\vee$ is
the unipotent radical of $B^\vee.$ This Peterson's (unpublished)
result was verified for homogeneous varieties $G/P$ of Lie type $A$
(\cite{Riet1}, \cite{Riet2}), and for even and odd  orthogonal
Grassmannians (\cite{CH1}). On the other hand,
$LG(n)=Sp_{2n}(\C)/P_n$ is not minuscule but cominuscule, but still
we  can find a variety $\mathcal{V}_{P_n}\subset SO_{2n+1}(\C)$,
defined in the same way as in the minuscule case, of which the
coordinate ring $\mathcal{O}(\mathcal{V}_{P_n})$ turns out to be
isomorphic with the quantum cohomology ring of $\lg$(\cite{CH1}).

\subsection{Varieties $\mathcal{V}_n$ and $\mathcal{W}_n$}
Note that the Peterson variety $\mathcal{V}_{P_n}$ was defined Lie
theoretically, and so we cannot see how the coordinate ring of
$\mathcal{V}_{P_n}$ looks like directly from the definition of
$\mathcal{V}_{P_n}$. To avoid some complexity, here we will not give
the definition of  $\mathcal{V}_{P_n}$.  Instead, we will give a
`unraveled' version of  $\mathcal{V}_{P_n}$ for our cases which
serves our purpose better.

 For $V_1,...,V_n\in \C$, let  $\tilde{v}(V_1,...,V_n)$ be the
matrix in $SL_{2n}(\C)$ \be \label{eltV} \tilde{v}(V_1,...,V_n):=
\left(\begin{array}{cccccccc}
1&V_1&V_2&\cdots &V_n&0&\cdots&0\\
&1&V_1&V_2& \cdots&V_n&\dd&\vd\\
& &1&V_1& & &\dd&0\\
&&&\dd&&&&V_n\\
&&&&\dd&&&\vd\\
&&&&&1&V_1&V_2\\
&&&&&&1&V_1\\
&&&&&&&1
\end{array}\right).
\ee For $i=1,...,n-1,$  put \be \label{Vii}
V_{i,i}:=V_i^2+2\sum_{k=1}^i(-1)^k V_{i+k}V_{i-k},\ee
 where
$V_0=1$, and $V_r=0$ if $r\geq n+1.$
\begin{lemma}(\cite{CH1})\label{Lemma:SL:Sp}
If $\tilde{v}(V_1,...,V_n)$ is an element of $ SL_{2n}(\C)$ such
that  $V_{i,i}=0$ for $i=1,...,n-1$,
 then the matrix $\tilde{v}(V_1,...,V_n)$ in fact belongs to $Sp_{2n}(\C)$.
 \end{lemma}

Lemma \ref{Lemma:SL:Sp} makes the following definition well-defined.
\begin{definition}
For $OG(n)=SO_{2n+1}(\C)/P_n,$ we define
$\mathcal{V}_n=\mathcal{V}_{P_n}$ to  be the subvariety of
$Sp_{2n}(\C)$ consisting of matrices  of the form
$\tilde{v}(V_1,...,V_n)$ satisfying the relations $V_{i,i}=0$ for
$i=1,...,n-1$.
\end{definition}
Note that if  we  view $V_i$  as coordinate functions of
$\mathcal{V}_n$, then the coordinate ring
$\mathcal{O}(\mathcal{V}_n)$ of $\mathcal{V}_n$ is
$\C[V_1,...,V_n]/\mathcal{I},$ where $\mathcal{I}$ is generated by
$V_{i,i}$ for all $i=1,...,n-1.$

\bigskip

For $W_1,...,W_{n+1}\in \C$, let  $\tilde{w}(W_1,...,W_{n+1})$ be
the matrix in $SL_{2n+1}(\C)$ \be\label{eltW}
\tilde{w}(W_1,...,W_{n+1}):= \left(\begin{array}{cccccccc}
1&W_1&W_2&\cdots &W_{n+1}&0&\cdots&0\\
&1&W_1&W_2& \cdots&W_{n+1}&\dd&\vd\\
& &1&W_1& & &\dd&0\\
&&&\dd&&&&W_{n+1}\\
&&&&\dd&&&\vd\\
&&&&&1&W_1&W_2\\
&&&&&&1&W_1\\
&&&&&&&1
\end{array}\right).
\ee

 For $i=1,...,n,$ let  \be W_{i,i}:=W_i^2+2\sum_{k=1}^i(-1)^k
W_{i+k}W_{i-k},\ee
 where
$W_0=1$,and $W_r=0$ if $r\geq n+2$.

\begin{lemma}(\cite{CH1})\label{lemma;SL:So}
If $\tilde{w}(W_1,...,W_{n+1})$ is an element of $ SL_{2n+1}(\C)$
such that  $W_{i,i}=0$ for $i=1,...,n$, then the matrix
$\tilde{w}(W_1,...,W_{n+1})$ in fact belongs to $SO_{2n+1}(\C).$
\end{lemma}

By Lemma \ref{lemma;SL:So}, the following definition makes sense.
\begin{definition} For $LG(n)=Sp_{2n}(\C)/P_n,$ we define
$\mathcal{W}_n=\mathcal{V}_{P_n}$ to  be the subvariety of
$SO_{2n+1}(\C)$ consisting of matrices  of the form
$\tilde{w}(W_1,...,W_{n+1})$ satisfying the relations $W_{i,i}=0$
for $i=1,...,n$.
\end{definition}

Note that if we view $W_i$ as coordinate functions of
$\mathcal{W}_n$, then the coordinate ring
$\mathcal{O}(\mathcal{W}_n)$ is $\C[W_1,...,W_{n+1}]/\mathcal{J},$
where $\mathcal{J}$ is generated by $W_{i,i}$ for all $i=1,...,n.$

\subsection{Comparing two presentation of the quantum cohomology ring}
 A Peterson's result for our cases can be stated as follows. See \cite{CH1} for an elementary proof.

\begin{theorem}[Peterson]\label{ppppp} We have isomorphisms of two rings.
\bn
\item The map $qH^*(OG(n),\C) \stackrel{\sim}{\rightarrow}
\mathcal{O}(\mathcal{V}_n)$ sending $\tau_i$ to $\frac{1}{2}V_i$ for
all $i\leq n$, and $q$ to $\frac{1}{4}V_n^2$ is an isomorphism.
\item The map $qH^*(\lg,\C) \stackrel{\sim}{\rightarrow}
\mathcal{O}(\mathcal{W}_n)$ sending $\sigma_i$ to $W_i$ for all
$i\leq n,$ and $q$ to $2W_{n+1}$ is an isomorphism.
 \en
\end{theorem}

\begin{notation}
\bn \item By the isomorphism
$qH^*(OG(n),\C)\stackrel{\sim}{=}\mathcal{O}(\mathcal{V}_n)$, each
$\tau \in qH^*(OG(n))$ defines a function on $\mathcal{V}_n$. We
denote this function by $\dot{\tau}=\dot{\tau}(V_1,...,V_n)$.

\item
Similarly, for $\sigma\in qH^*(LG(n)),$
$\dot{\sigma}=\dot{\sigma}(W_1,...,W_{n+1})$ denotes the function on
$\mathcal{W}_n$ corresponding to $\sigma$ under the isomorphism
$qH^*(LG(n))\stackrel{\sim}{=}\mathcal{O}(\mathcal{W}_n)$.\en
\end{notation}

\begin{example}
 For nonnegative integers $i\geq j$, let $V_{i,j}$ be the function on $\mathcal{V}_n$ defined by
$$V_{i,j}:=V_i
V_j+2\sum_{k=1}^{j}(-1)^k V_{i+k}V_{j-k},$$ where $V_0=1$ and
$V_l=0$ if $l<0$ or $l>n$. Then we have $V_{i,j}=4\dot{\tau}_{i,j}$
if $j\ne 0$, and if $j=0$ and $i\ne 0$, then
$V_{i,j}=V_i=2\dot{\tau}_i$. More generally, for $\lambda \in
\mathcal{D}(n)$,
 $\dot{\tau}_\lambda$ is
the function on $\mathcal{V}_n$  defined by
\begin{displaymath}
\dot{\tau}_{\lambda}=2^{-l}\textrm{Pfaffian}(V_{\lambda_i,\lambda_j})_{1\leq
i,j\leq r},
\end{displaymath}
where $r=2\lfloor(l+1)/2 \rfloor$ for $l=l(\lambda).$
\end{example}

\section{Analysis on points of Peterson's variety}

In this section, we record an explicit description of  elements of
$\mathcal{V}_n$  and $\mathcal{W}_n$ from Section $4$ of \cite{CH1}.

\subsection{Definitions and Notations.}\label{def and nota}
Let  $\zeta=\zeta_n$ be the primitive $2n$-th root of unity, i.e.,
$\zeta_n=e^{\frac{\pi i}{n}}.$ Let $\mathcal{T}_n$ be the set of all
$n$-tuples $J=(j_1,...,j_n),$ $-\frac{n-1}{2}\leq j_1<\cdots<j_n\leq
\frac{3n-1}{2}, $ such that $\zeta^J:=(\zeta^{j_1},...,\zeta^{j_n})$
is an $n$-tuple of distinct $2n$-th roots of $(-1)^{n+1}.$ Let us
call $I=(i_1,...,i_n)\in \mathcal{T}_n$ $exclusive$  if
$\zeta^{i_k}\ne -\zeta^{i_l}$ for all $k,l=1,...,n.$

Define subsets $\In$, $\In^e$ and $\In^o$ of $\mathcal{T}_n$  as
$$\mathcal{I}_n:=\{ I \in \mathcal{T}_n\hspace{0.05in}|\hspace{0.05in} I \hspace{0.05in}\mathrm{exclusive}\},$$
$$\In^e:=\{ I \in \In \hspace{0.05in}|\hspace{0.05in} E_n(\zeta^I)=1\},$$
$$\In^o:= \{ I \in \In \hspace{0.03in}|\hspace{0.03in} E_n(\zeta^I)=-1\}.$$

\begin{remark}
We can easily check that $|\mathcal{I}_n|=2^n=|\mathcal{D}(n)|.$
Note that $\In=\In^e \sqcup \In^o$ since $E_n^2(\zeta^I)=1$ for
$I\in \In$ by Lemma \ref{Lemma:Exclusive} below. Since
$|\In^e|=|\In^o|,$ it follows that $|\In^e|=|\In^o|=2^{n-1}$.
\end{remark}

Now we characterize exclusive $n$-tuples in terms of elementary
symmetric functions.
\begin{lemma}\label{Lemma:Exclusive}
If $I$ is exclusive, then $E_i(\zeta^{2I})=0$ for $i=1,2,...n-1$ and
$E_n^2(\zeta^I)=1.$
\end{lemma}

\begin{proof}
Note that $I_0:=(-\frac{n-1}{2},-\frac{n-1}{2}+1,...,\frac{n-1}{2})$
is exclusive, $E_i(\zeta^{2I_0})=0$ for $i=1,...,n-1$ and
$E_n(\zeta^{2I_0})=1.$ If $I$ is exclusive, we can easily check that
$\zeta^{2I}=\zeta^{2I_0}$. Thus $E_i(\zeta^{2I})=0$ for
$i=1,...,n-1,$ and $E_n^2(\zeta^I)=E_n(\zeta^{2I})=1$ by $(1)$ of
Proposition \ref{property-of-Qt-poly}.
\end{proof}

 We now characterize the elements of
$\mathcal{V}_n$ and $\mathcal{W}_n$ more explicitly. To do this, we
introduce the following notations.
\begin{notation}
\bn \item For $a_1,...,a_{n} \in \C,$ let $v(a_1,...,a_n)$ be the
matrix in $SL_{2n}(\C)$ defined  by
$$v(a_1,...,a_n)=\tilde{v}(E_1(a_1,...,a_n),...,E_n(a_1,...,a_n)),$$ where $\tilde{v}$ was given in (\ref{eltV}).
\item For $b_1,...,b_{n+1}\in \C,$
 let $w(b_1,...,b_{n+1})$ be the matrix in $SL_{2n+1}(\C)$ defined by
$$w(b_1,...,b_{n+1})=\tilde{w}(E_1(b_1,...,b_{n+1}),...,E_{n+1}(b_1,...,b_{n+1})),$$
where $\tilde{w}$ was given in (\ref{eltW}).\en
\end{notation}

\begin{proposition}(\cite{CH1}, Lemmas $5.2$, $5.3$)\label{pp} Elements of $\mathcal{V}_n$ and $\mathcal{W}_n$ are characterized as follows.
\bn
\item All elements of $\mathcal{V}_n$ are exactly of the form
$v(t\zeta^I)$ with
 $t\in \C$ and $I \in \In.$
 \item
 All elements of $\mathcal{W}_n$ are exactly of the form
$w(t\zeta^I)$ with $t\in \C$ and $I \in \mathcal{I}_{n+1}$. \en
\end{proposition}

\begin{remark}
 Note that for $\lambda \in \Dn$, the function $\dot{\tau}_{\lambda}$
evaluates on $v(t\zeta^I)\in \mathcal{V}_n$  to
$\Pt_\lambda(t\zeta^I)$, and  $\dot{\sigma}_{\lambda}$ evaluates on
$w(t\zeta^I) \in \mathcal{W}_n$  to $\Qt_\lambda(t\zeta^I)$).
\end{remark}

 For later use, here we record the evaluations on
$\mathcal{V}_n$ and $\mathcal{W}_n$ of $\dot{q}$ for the quantum
variable $q$ for $OG(n)$ and $LG(n).$

\begin{proposition}(\cite{CH1}, Lemma $5.5$)\label{lemma:quant12}
 The functions $\dot{q}$ on $\mathcal{V}_n$ and $\mathcal{W}_n$ evaluate as follows.
\bn
\item For
$v=v(t\zeta^I)\in \mathcal{V}_n,$ we have
$$\dot{q}(v)=\frac{1}{4}t^{2n}.$$
\item For
$w=w(t\zeta^I)\in \mathcal{W}_n,$ we have
$$\dot{q}(w)=2t^{n+1}E_{n+1}(\zeta^I), \hspace{0.06in}
\textrm{and}\hspace{0.1in} \dot{q}^2(w)=4t^{2n+2}.$$\en
\end{proposition}

\begin{definition} Let $\mathcal{V}_n^\prime$  (resp. $\mathcal{W}_n^\prime$) be the subvariety
 of $\mathcal{V}_n$   (resp. $\mathcal{W}_n$)
 defined by the function $\dot{q}=1.$
 \end{definition}

\begin{corollary}\label{cor:expression}
Let $\epsilon=\epsilon_n=(4)^{\frac{1}{2n}}.$ Then $v^\prime \in
\mathcal{V}_n^\prime$ if and only if there exists a unique $I\in
\mathcal{I}_n$ such that $v^\prime=v(\epsilon \zeta^I)$.
 \end{corollary}

\begin{proof}
The direction $(\Leftarrow)$ is obvious. For the converse, let
$v^\prime \in \mathcal{V}_n^\prime $. Then, by Proposition \ref{pp},
$v^\prime$ can be written as $v^\prime=v(t\zeta^J)$ for some $t\in
\C$ and $J\in \mathcal{I}_n.$ Note that
  $\dot{q}(v^\prime)=\frac{1}{4}t^{2n}=1$ since $v^\prime \in
\mathcal{V}_n^\prime $.  Now, since $\frac{t}{\epsilon}$ is a
$2n$-th root of unity,  and $\frac{t}{\epsilon}\zeta^J$ is an
$n$-tuple $\zeta^J$ rotated  by $\mathrm{arg}(\frac{t}{\epsilon})$
in each entry, there is a unique $I\in \In$ such that
$\zeta^I=\frac{t}{\epsilon}\zeta^J$, i.e., $\epsilon \zeta^I=t
\zeta^J$ . This proves the corollary.
\end{proof}

 \begin{corollary}Let
 $\delta=\delta_{n}:=(\frac{1}{2})^{\frac{1}{n+1}}$.  Then $w^\prime \in
\mathcal{W}_n^\prime$ if and only if there exists a unique $I\in
\mathcal{I}_{n+1}^e$ such that $w^\prime=v(\delta\zeta^I)$.
\end{corollary}
\begin{proof}
The proof is similar to the proof of Corollary \ref{cor:expression}.
\end{proof}

\subsection{Orthogonality formulas}
\begin{lemma}\label{lemma:action}
For $I, J\in \In$ with $I\ne J,$ there is $w\in W_n\setminus S_n$
such that $(\zeta^I)^w=\zeta^J$.
\end{lemma}

\begin{proof} Write $I=(i_1,...,i_n)$ and $J=(j_1,...,j_n)$. Let us
determine (entries $w_i$ of) $w=(w_1,...,w_n)\in W_n\setminus S_n.$
First note that since $I$ is exclusive, the union of two sets
$\{\zeta^{i_1},...,\zeta^{i_n}\}\cup
\{-\zeta^{i_1},...,-\zeta^{i_n}\} $ equals the set of all $2n$-th
roots of $(-1)^{n+1}.$ Therefore for each $k=1,...,n$, the entry
$\zeta^{j_k}$ of $\zeta^{J}$ is $\zeta^{i_m}$ or $-\zeta^{i_m}$ for
some $1\leq m\leq n.$ Then put $w_m=k$ (resp. $\bar{k}$) if
$\zeta^{j_k}=\zeta^{i_m}$ (resp. $\zeta^{j_k}=-\zeta^{i_m}$). Then
for $w=(w_1,...,w_n)$ thus obtained, it is obvious that
$(\zeta^I)^w=\zeta^J$, and $w\in W_n\setminus S_n$ since
$\{\zeta^{i_1},...,\zeta^{i_n}\}\ne
\{\zeta^{j_1},...,\zeta^{j_n}\}$.
\end{proof}

\begin{proposition}\label{p}
\bn
\item
 For $I,J \in \mathcal{I}_n$ and $t\in \C,$ we have
$$\sum_{\lambda \in \mathcal{D}(n)}\Pt_{\lambda}(t\zeta^I)\Pt_{
\hat{\lambda}}(t\zeta^J)=\delta_{I,J}S_{\rho_n}(t\zeta^I),$$
\item
For $I,J \in \mathcal{I}_{n+1}^e$ and $t\in \C,$ we have
\be\label{formula} \sum_{\lambda \in
\mathcal{D}(n)}t^{n+1}\Qt_{\lambda}(t\zeta^I)\Qt_{
\hat{\lambda}}(t\zeta^J)=\delta_{I,J}2^{n}S_{\rho_{n+1}}(t\zeta^I),\ee
\en
 where we denote $\rho_n:=(n,n-1,...,1)$.
 \end{proposition}

\begin{proof}
By Lemma \ref{lemma:action}, for $I, J\in \In$ there is $w\in W_n
\setminus S_n$ such that
$(\zeta^{i_1},...,\zeta^{i_n})^w=(\zeta^{j_1},...,\zeta^{j_n}).$
Therefore $(1)$ is immediate from Proposition \ref{Key}. The
equality $(2)$ follows from the equalities

$$2^{n+1}S_{\rho_{n+1}}(t\zeta^I)=\sum_{\lambda \in \mathcal{D}(n+1)}\Qt_{\lambda}(t\zeta^I)\Qt_{
\hat{\lambda}}(t\zeta^I)$$
$$=\sum_{\mu \in \mathcal{D}(n)}2\Qt_{n+1}(t\zeta^I)\Qt_{\mu}(t\zeta^I)\Qt_{
\hat{\mu}}(t\zeta^I)=\sum_{\mu \in
\mathcal{D}(n)}2t^{n+1}E_{n+1}(\zeta^I)\Qt_{\mu}(t\zeta^I)\Qt_{
\hat{\mu}}(t\zeta^I).$$ Here the first equality follows from $(1)$
of Proposition \ref{p}, and the second equality follows from the
fact that for each pair $(\mu, \hat{\mu})\in \mathcal{D}(n)\times
\mathcal{D}(n),$ there are exactly two pairs
$(\lambda,\hat{\lambda})\in \mathcal{D}(n+1)\times
\mathcal{D}(n+1)$, i.e., $\lambda=(n+1,\mu)$ (and so
$\hat{\lambda}=(\hat{\mu},0)$), or $\hat{\lambda}=(n+1,\hat{\mu})$
(and so $\lambda=(\mu,0)$), for each of which we have
$$\Qt_{\lambda}(t\zeta^I)\Qt_{
\hat{\lambda}}(t\zeta^I)=\Qt_{n+1}(t\zeta^I)\Qt_{\mu}(t\zeta^I)\Qt_{\hat{\mu}}(t\zeta^I).$$
Since $I\in \mathcal{I}_{n+1}^e,$ i.e., $E_{n+1}(\zeta^I)=1,$ we get
$(2).$
\end{proof}

\section{Quantum Multiplication Operators}
In this section, we shall give a description of eigenvalues and
eigenvectors of multiplication operators on $qH^*(M(n),\C)_{q=1}$.
The following lemma will be used in Theorems \ref{Main-Thm1} and
\ref{Main-Thm2}.

\begin{lemma}\label{nonidentity}
\bn
\item For $S_{\rho_n}(x_1,...,x_n),$ $S_{\rho_n}(\zeta^I)\ne 0$  for any $I\in \In$.
\item For $S_{\rho_n}(x_1,...,x_{n+1}),$ $S_{\rho_n}(\zeta^J)\ne 0$  for any $J\in \mathcal{I}_{n+1}$.

\en
\end{lemma}

\begin{proof}
 Denote by $\dot{e}_q$ the function on
$\mathcal{V}_n^\prime$ corresponding to the quantum Euler class
$e_q(OG(n))$. For $I\in \In$, evaluate $\dot{e}_q$ on the point
$v=v(\epsilon\zeta^I)$: From the definition of $e_q(OG(n))$ in $\S$
\ref{subsection-quan-euler}
 and $(1)$ of Proposition \ref{p}, we have
$$\dot{e}_q(v(\epsilon\zeta^I))=\sum_{\lambda \in \mathcal{D}(n)}\Pt_{\lambda}(\epsilon\zeta^I)\Pt_{
\hat{\lambda}}(\epsilon\zeta^J)=S_{\rho_n}(\epsilon \zeta^I).$$
Recall that $OG(n)$ is minuscule and hence the ring
$qH^*(OG(n)_{q=1}$ is semisimple (Proposition \ref{semisimplicity}).
Thus it follows from Proposition \ref{euler cri} that
$S_{\rho_n}(\epsilon \zeta^I)\ne 0,$ equivalently, $S_{\rho_n}(
\zeta^I)\ne 0$  for $I\in \In$, which proves $(1)$ of the lemma.
Similarly,  using the semisimplicity of $qH^*(LG(n))_{q=1}$ and
$(2)$ of Proposition \ref{p}, we get $S_{\rho_n}(\zeta^J)\ne 0$ for
$J\in \mathcal{I}_{n+1}.$  This proves the lemma.
\end{proof}

\begin{prelemma}\label{prelemma}
Let $\{v_0, v_1,....,v_m\}$ and $\{v_0^\prime,
v_1^\prime,...,v_m^\prime\}$ be the sets of nonzero vectors in the
(closed) upper half plane of the complex plane satisfying
\bn[label=(\roman*)]
\item $v_0=v_0^\prime$ and $v_0\in \R_{>0},$
\item $|v_k|=|v_k^\prime|$ for $k=1,...,n$,
\item
$0<\mathrm{arg}(v_1)<\cdots <\mathrm{arg}(v_m)\leq \pi$; and $0<
\mathrm{arg}(v_1^\prime)<\cdots <\mathrm{arg}(v_m^\prime)\leq \pi$,
\item $\mathrm{arg}(v_{k+1}^\prime)-\mathrm{arg}(v_k^\prime)\leq \mathrm{arg}(v_{k+1})-\mathrm{arg}(v_k)$
for all $k=0,...,m-1$. \en Then we have $$|\sum_{i=1}^m
v_i^\prime|\geq |\sum_{i=1}^m v_i|,
\hspace{0.1in}\mathrm{arg}(\sum_{i=1}^m v_i)\geq
\mathrm{arg}(\sum_{i=1}^m v_i^\prime),$$  and the strict
inequalities hold if   there is a $k$ with $1\leq k\leq m$ such that
$\mathrm{arg}(v_k)>\mathrm{arg}(v_k^\prime)$.
\end{prelemma}

\begin{proof}
Obvious from the definitions of the length and summation of vectors
in the complex plane.
\end{proof}

The following is more general than Prelemma \ref{prelemma}
\begin{prelemma}\label{prelemma:corollary}
Let $\{v_{-l},...,v_{-1},v_0, v_1,....,v_m\}$ and
$\{v_{-l}^\prime,...,v_{-1}^\prime ,v_0^\prime,
v_{1}^\prime,....,v_m^\prime\}$ be  the sets of distinct nonzero
vectors in the complex plane satisfying \bn[label=(\roman*)]
\item  $v_k$ and $v_k^\prime$ lie on  the (closed) upper half plane if $k=0, 1,...,m$, and  on the (open) lower half plane if $k=-1,...,-l.$
\item The vectors $v_k,$ and $v_k^\prime$ in the (closed) upper half plane satisfy the conditions  in Prelemma
\ref{prelemma}.
\item The vectors $v_k,$ and $v_k^\prime$ in the lower half plane satisfy the similar conditions;
   \bn
\item  $|v_k|=|v_k^\prime|$ for $k=-1,...,-l$,
\item $-\pi < \mathrm{arg}(v_{-l})<\cdots <\mathrm{arg}(v_{-1})<0$; and $-\pi < \mathrm{arg}(v_{-l}^\prime)<\cdots <\mathrm{arg}(v_{-1}^\prime)<0$,
\item $\mathrm{arg}(v_{k})-\mathrm{arg}(v_{k-1})\geq \mathrm{arg}(v_{k}^\prime)-\mathrm{arg}(v_{k-1}^\prime)$
for all $k=0,...,-1+1$.\en \en
 Then we have $|\sum_{k=-l}^m v_k|\leq
|\sum_{k=-l}^m v_k^\prime|,$
    and the strict inequality holds if there is a $k$  with $-l\leq k\leq m$ such that
    $|\mathrm{arg}(v_k)|>|\mathrm{arg}(v_k^\prime)|$.
\end{prelemma}

\begin{proof}
When $l\leq 1$ and $m \leq 1,$  it is trivial to check the prelemma.
Furthermore,  when  $m=1$ and $l=1$, i.e., when we work with the two
sets $\{v_{1},v_0,v_1\}$ and
$\{v_{1}^\prime,v_0^\prime,v_1^\prime\}$, the prelemma is true even
though we relax the condition $|v_k|=|v_k^\prime|$ into the
condition $|v_k|\leq |v_k^\prime|$ for $k=-1,1$. For general case,
 let $V_1:=\sum_{k=1}^m v_k$ (resp.  $V_1^\prime:=\sum_{k=1}^m v_k^\prime$)
  and $V_{-1}:=\sum_{k=1}^lv_{-k}$ (resp. $V_{-1}^\prime:=\sum_{k=1}^lv_{-k}^\prime$).
  Then, by Prelemma \ref{prelemma}, the vectors in the sets
  $\{V_{-1},v_0,V_1\}$ and  $\{V_{-1}^\prime,v_0,V_1^\prime\}$ satisfy all the
  conditions of the prelemma (for $m=1$ and $l=1$)
  except for the condition  $|V_{k}|= |V_{k}^\prime|$ for $k=-1,1$.
  Instead, they  satisfy the condition $|V_k|\leq |V_k^\prime|$ for $k=-1, 1.$
 Then, by the above special case, we have
 $|V_{-1}+v_0+V_1| \leq |V_{-1}^\prime+v_0^\prime+V_1^\prime|$,
 equivalently, $|\sum_{k=-l}^m v_k|\leq \sum_{k=-l}^m v_k^\prime|.$
\end{proof}

 Recall that the entries $\zeta^{i_1},...,\zeta^{i_n}$ of  $\zeta^I$ lie  on the unit circle.
 Therefore one can rotate the points $\zeta^{i_1},...,\zeta^{i_n}$ simultaneously
 by a certain angle $\theta$  so that the set of rotated points
 $\{e^{2\pi \theta}\zeta^{i_1},..., e^{2\pi \theta}\zeta^{i_n}\}$ equals
 the set  $\{\zeta^{j_1},...,\zeta^{j_n}\}$ for some $J=(j_1,...,j_n)\in\In.$ In this situation,
 we simply say that $\zeta^J$ is obtained by rotating $\zeta^I$ (by $\theta$).
 Similarly, one can flip the points $\zeta^{i_1},...,\zeta^{i_n}$ simultaneously
 with respect to a line $L$ passing through the origin, so that the set of flipped points
 $\{(\zeta^{i_1})^{\prime},...,(\zeta^{i_n})^\prime\}$ equals the set
 $\{\zeta^{j_1},...,\zeta^{j_n}\}$ for some $J=(j_1,...,j_n)\in\In.$
 In this situation, we simply say that $\zeta^J$ is obtained by flipping $\zeta^I$ (with respect to $L$).

 \begin{definition}
 \bn
 \item
 We say that for $I,J \in \In,$  $\zeta^I$ has the {\it same configuration}
 as $\zeta^J$  if  $\zeta^J$ is  obtained by rotating and (or) flipping $\zeta^I.$
 \item For  $J\in \In,$
 $\zeta^J$ is called $closed$ if
$j_{k+1}=j_k+1$ for $k=1,...,n-1.$ For example, $\zeta^{I_0}$ is
closed. \en
 \end{definition}

\begin{lemma}\label{eigenvalue}For the $n$-tuples $\zeta^I$ with $I\in
\In$, we have the following properties.

\bn
\item If $\zeta^I$ and $\zeta^J$ have the same configuration, then $|E_1(\zeta^I)|=|E_1(\zeta^J)|$.
\item
$E_1(\zeta^{I_0})$ is a positive real number which is equal to
$E_1(\zeta^{I_0})=\frac{1}{\mathrm{sin}(\pi/2n)}.$
\item  $\zeta^J$ is a closed $n$-tuple if and only if $|E_1(\zeta^{J})|$
is maximal among  $|E_1(\zeta^I)|$ with $I\in \In.$ In particular,
$|E_1(\zeta^{I_0})|=E_1(\zeta^{I_0})$ is maximal among
$|E_1(\zeta^I)|$ with $I\in \In. $
\item If $\zeta^I$ and $\zeta^J$ are closed $n$-tuples with $I\ne J \in
\mathcal{I}_{n}$, then we have $E_1(\zeta^I)=\eta E_1(\zeta^J)$  for
some $2n$-th root $\eta$ of unity with $\eta \ne 1,$ and, in
particular,  $E_1(\zeta^I)\ne E_1(\zeta^J)$. \en
\end{lemma}

\begin{proof}
$(1)$ is obvious. For $(2)$, see Page $542$ of \cite{Riet1}. For
$(3)$, first note that for any $\theta\in \R,$ $|E_1(\zeta^J)|$ is a
maximal element of the set $\{|E_1(\zeta^{I})|\hspace{0.05in}
|\hspace{0.05in} I\in \mathcal{I}_n\}$ if and only if $|E_1(e^{2\pi
\textbf{i}\theta} \zeta^{J})|$ is a  maximal element of the set
$\{|E_1(e^{2\pi \textbf{i}\theta}\zeta^{I})| \hspace{0.05in}|
\hspace{0.05in} I\in \mathcal{I}_n\}$. Let $\zeta^J$ be a closed
$n$-tuple with $J=(j_1,...,j_n)$. Then, fix a component
$\zeta^{j_k}$ of the $n$-tuple $\zeta^{J}$ and take $\theta\in \R$
so that the vector $v_0:=e^{2\pi \textbf{i}\theta}\zeta^{j_k}$ lies
on the positive real axis. Now we consider the set  $\{|E_1(e^{2\pi
\textbf{i}\theta}\zeta^{I})| \hspace{0.05in}| \hspace{0.05in} I\in
\mathcal{I}_n\}$. Then by Prelemma \ref{prelemma:corollary},
$|E_1(e^{2\pi \textbf{i}\theta}\zeta^{J})$ is a maximal element of
the set $\{|E_1(e^{2\pi \textbf{i}\theta}\zeta^{I})|
\hspace{0.05in}| \hspace{0.05in} I\in \mathcal{I}_n\}$, and hence
$|E_1(\zeta^J)|$ is a maximal element of the set
$\{|E_1(\zeta^{I})|\hspace{0.05in} |\hspace{0.05in} I\in
\mathcal{I}_n\}$. The converse is immediate from Prelemma
\ref{prelemma:corollary}. For $(4),$ note that $\zeta^I$ and
$\zeta^J$ are closed $n$-tuples with $I\ne J,$ then there is a
$2n$-th root $\eta\ne 1$ of unity such that $\zeta^I=\eta \zeta^J$,
and hence $E_1(\zeta^I)=\eta E_1(\zeta^J).$
\end{proof}

\begin{remark}
Note that the converse of $(1)$ is not true in general. Indeed, it
is not difficult to find  $I,J\in \In$ such that $\zeta^I$ and
$\zeta^J$ do not have the same configuration, and
$|E_1(\zeta^I)|=|E_1(\zeta^J)|$. However, $(3)$ implies that with
the maximality condition on the modulus, this can not happen.
\end{remark}

We remark that the ring $qH^*(OG(n),\C)_{q=1}$ (resp.
$qH^*(LG(n),\C)_{q=1}$) is a $2^n$ dimensional complex vector space
with the Schubert basis $\{\tau_{\lambda}
\hspace{0.04in}|\hspace{0.04in} \lambda\in \Dn\}$ (resp.
$\{\sigma_{\lambda} \hspace{0.04in}|\hspace{0.04in} \lambda\in
\Dn\}$). Now we use the orthogonality formulas in Proposition
\ref{p} to find another basis for each of these vector spaces which,
in fact, turns out to be a simultaneous eigenbasis for all
multiplication operators $[F]$. The original idea for finding this
eigenbasis by using the orthogonality formula is due to Rietsch.
Indeed, Rietsch obtained an eigenbasis of the quantum cohomology
ring of the Grassmannian which Theorems \ref{Main-Thm1} and
\ref{Main-Thm2} below are modeled on (Page $551$ of \cite{Riet1}).

\begin{theorem}\label{Main-Thm1}
For each $I\in \In,$ let $\tau_I=\sum_{\nu \in
\mathcal{D}(n)}\Pt_\nu(\epsilon\zeta^I))\tau_{\hat{\nu}}.$ Then for
each $\lambda \in \mathcal{D}(n)$,
 the quantum multiplication operator $[\tau_\lambda]$ on  $qH^*(OG(n),\C)_{q=1}$ has
 eigenvectors $\tau_I$ with eigenvalues $\Pt_\lambda(\epsilon\zeta^I).$ In particular,
 $\{\tau_I\hspace{0.03in}|\hspace{0.03in}I \in \mathcal{I}_n\}$
 forms
a simultaneous eigenbasis of the vector space $qH^*(OG(n),\C)_{q=1}$
for the operators $[F]$ with  $F \in qH^*(\og,\C)_{q=1}$. Here
$\epsilon=(4)^{\frac{1}{2n}}$ as before.
\end{theorem}

\begin{proof}
  First note that $\tau_I$ is a nonzero vector for all $I\in \In$  by Lemma
  \ref{nonidentity} and $(1)$ of Proposition \ref{p}.
Evaluating the function $\dot{\tau}_\lambda\cdot \dot{\tau}_I$ on
the points
 $v(\epsilon\zeta^J)$ for $J\in \In$ and using $(1)$ of Proposition \ref{p},
 we obtain the equality of
 functions
 on $\mathcal{V}_n^\prime$
 \be \label{tau-operator} \dot{\tau}_\lambda\cdot \dot{\tau}_I
 =\Pt_\lambda(\epsilon\zeta^I)\dot{\tau}_I.\ee
 Then the first part of the theorem is obvious since
 $qH^*(OG(n))_{q=1}$ is identified with $\mathcal{O}(\mathcal{V}^\prime_n).$
 Since (\ref{tau-operator}) holds for all Schubert basis elements $\tau_\lambda$,
 the vector $\tau_I$ is a simultaneous eigenvector for all operators $[F]$
 with an eigenvalue $\dot{F}(v(\epsilon\zeta^I))$ for $F\in qH^*(OG(n),\C)_{q=1}.$
 Now we show
that $\{\tau_I\hspace{0.03in}|\hspace{0.03in}I \in \mathcal{I}_n\}$
forms a simultaneous eigenbasis for $ qH^*(OG(n),\C)_{q=1}.$ Since
$|\In|=|\mathcal{D}|$,  it suffices to show that the vectors
$\tau_I,$ $I\in \In,$ are linearly independent. So suppose that
$$\sum_{I\in \In}a_I \tau_I=0.$$ Now let us evaluate the function
$\dot{\tau}:=\sum_{I\in \In}a_I \dot{\tau}_I$ on the points
$v(\epsilon\zeta^J)$ with $J\in \mathcal{I}_n$. Then, by $(1)$ of
Proposition \ref{p}, we have the evaluation
$$\dot{\tau}(v(\epsilon\zeta^J))=a_J\dot{S}_{\rho_n}(v(\epsilon\zeta^J))=
a_JS_{\rho_n}(\epsilon\zeta^J).$$ Since
$S_{\rho_n}(\epsilon\zeta^J)$ is nonzero  for any $J\in \In$ by
Lemma \ref{nonidentity}, $a_J=0$ for each $J\in \In.$ Therefore the
vectors $\tau_I$, $I \in \mathcal{I}_n$, are linearly independent.
This completes the proof.
\end{proof}

\begin{theorem}\label{Main-Thm2}
For each $I\in \mathcal{I}_{n+1}^e,$ let $\sigma_I=\sum_{\nu \in
\mathcal{D}(n)}\Qt_\nu(\delta\zeta^I))\sigma_{\hat{\nu}}.$ Then for
each $\lambda \in \mathcal{D}(n)$,
 the quantum multiplication operator $[\sigma_\lambda]$ on
 $qH^*(\lg,\C)_{q=1}$ has  eigenvectors $\sigma_I$ with eigenvalues
 $\Qt_\lambda(\delta\zeta^I).$ In particular,
  $\{\sigma_I \hspace{0.04in}|\hspace{0.04in}I \in \mathcal{I}^e_{n+1}\}$
 forms a simultaneous eigenbasis of the vector space $qH^*(\lg,\C)_{q=1}$
for the operators $[F]$ with  $F \in qH^*(\lg,\C)_{q=1}$. Here
$\delta=(\frac{1}{2})^{\frac{1}{n+1}}$ as before.
\end{theorem}
\begin{proof}
From $(2)$ of Proposition \ref{p}, we have $$\sum_{\lambda \in
\mathcal{D}(n)}\Qt_{\lambda}(\delta\zeta^I)\Qt_{
\hat{\lambda}}(\delta\zeta^J)=\delta_{I,J}2^{n}S_{\rho_{n+1}}(\delta\zeta^I)
\hspace{0.08in} \mathrm{for} \hspace{0.04in}I\in
\mathcal{I}_{n+1}^e,$$ from which we get eigenvectors $\sigma_I$ and
corresponding eigenvalues $\Qt_{\lambda}(\delta\zeta^I)$ \be
\label{sigma-operator} \sigma_\lambda\cdot
\sigma_I=\Qt_\lambda(\delta\zeta^I)\sigma_I,\ee as in the case of
$\og.$ The proof of linear independence of vectors $\sigma_I,$ $I\in
\mathcal{I}_{n+1}^e,$ is the same as in the case of $\og.$
\end{proof}

 For a complex manifold, the $i$-th Chern class $c_i(M)$ is defined
to be the $i$-th Chern class $c_i(TM)$ of the tangent bundle $TM.$

\begin{lemma}( \cite{FW}, Lemma $3.5$)\label{first-chern}  The first Chern classes of $\og$ and $\lg$, respectively, are given by
$$ c_1(\og)=2n\tau_1\hspace{0.2in}\textrm{and}\hspace{0.2in}
c_1(\lg)=(n+1)\sigma_1.$$
\end{lemma}

  For $I\in \In$, let $$f(I):=n \epsilon E_1(\zeta^I)=2\dot{\tau}_1(v(\epsilon\zeta^I)).$$ Then since $c_1(OG(n))=2n\tau_1,$  by Theorems
\ref{Main-Thm1}, $f(I)$ is the eigenvalue of the operator
$[c_1(OG(n))]$ corresponding to the eigenbasis $\tau_I$. Therefore,
we have
$$\mathrm{Spec}\hspace{0.03in} ([ c_1(OG(n))])=\{\hspace{0.04in}f(I)\hspace{0.05in}|
\hspace{0.05in}I\in \In\}.$$

Similarly, for $I\in \mathcal{I}_{n+1}^e$, letting
$$g(I):=(n+1)\delta E_1(\zeta^I)=(n+1)\dot{\sigma}_1(w(\delta
\zeta^I)),$$  we have  $$ \mathrm{Spec}\hspace{0.03in}
([c_1(LG(n))])= \{\hspace{0.04in}g(I)\hspace{0.05in}|\hspace{0.05in}
I\in \mathcal{I}^e_{n+1}\}.$$

\begin{lemma}\label{lemma:multi}For $I\in \In$ for $OG(n)$ (resp. $I\in
\mathcal{I}_{n+1}^e$ for $LG(n)$), the multiplicity of the
eigenvalue $f(I)$  (resp. $g(I)$) is equal to the cardinality of the
set $\{ J\in \mathcal{I}_n \hspace{0.06in}|
\hspace{0.06in}f(J)=f(I)\}$ ( resp. $\{ J\in \mathcal{I}_{n+1}^e
\hspace{0.06in}| \hspace{0.06in} g(J)=g(I)\}$). In particular, if an
eigenvalue $f(I)$ (resp. $g(I)$) has a maximal modulus among the
eigenvalues for $OG(n)$ (resp. $LG(n) $), then $f(I)$ (resp. $g(I)$)
is a simple eigenvalue.
\end{lemma}

\begin{proof}
The first statement is obvious, and the second is immediate from
$(3), (4)$ of Lemma \ref{eigenvalue}.
\end{proof}

\begin{theorem}\label{Coro:eigenvalue}
 $\og$ and $\lg$ satisfy Conjecture $\mathcal{O}$.
\end{theorem}
\begin{proof}
By $(3)$ of Lemma \ref{eigenvalue}, the operator $[ c_1(OG(n))]$ has
the eigenvalue $T_0:=f(I_0)$, which is a positive real number of
maximal modulus among $f(I)$ with $I\in \In.$ Therefore the
condition $(1)$ of Conjecture $\mathcal{O}$ is satisfied by $OG(n)$.
For the same reason, the condition $(1)$ holds for $LG(n).$ The
condition $(3)$ of Conjecture $\mathcal{O}$ for both cases follows
from Lemma \ref{lemma:multi} and $(3), (4)$ of Lemma
\ref{eigenvalue}.

\smallskip
For the condition $(2)$  for $OG(n)$, suppose that $J\in \In$ is
such that  $|f(J)|=f(I_0)$. Then $|E_1(J)|=E_1(I_0)$, and so
$|E_1(J)|$ is maximal among $|E_1(\zeta^I)|$ with $I\in \In$. Then
by $(3)$, $(4)$ of Lemma \ref{eigenvalue}, there is a $2n$-th root
$\eta$ of unity such that $\zeta^J=\eta\zeta^{I_0}$, equivalently,
$E_1(\zeta^J)=\eta E_1(\zeta^{I_0})$. Thus $f(J)=\eta f(I_0).$  But
since the Fano index of $OG(n)$ is $2n,$ the condition $(2)$ for
$OG(n)$ is immediate.

\smallskip
 For
the condition $(2)$ for $LG(n)$, suppose that $J\in
\mathcal{I}_{n+1}^e$ is such that $|g(J)|=g(I_0)$. Then,  as above,
applying $(4)$ of Lemma \ref{eigenvalue}   to $J\in
\mathcal{I}_{n+1}^e$, there is a $(2n+2)$-root $\xi$   of unity such
that $E_1(\zeta^J)=\xi E_1(\zeta^{I_0})$, and so $g(J)=\xi g(I_0)$.
Note that, a priori, $\xi$ is not necessarily a $(n+1)$-th root of
unity.  But since both $I_0$ and $J$  belong to
 $\mathcal{I}^e_{n+1},$  $\xi$ is in fact a $(n+1)$-th root of unity.  Since the Fano index of $LG(n)$ is $n+1$,
  it follows that the condition $(2)$ is satisfied  by $LG(n)$.
 This completes the proof.
\end{proof}

\end{document}